\documentclass[11pt,centertags,reqno,twoside]{amsart}
\usepackage{amssymb}
\usepackage{mathptmx}

\DeclareSymbolFont{SY}{U}{psy}{m}{n}
\DeclareMathSymbol{\emptyset}{\mathord}{SY}{'306}

\newcommand{\cal}{\mathcal}
\newcommand{\bbC}{{\mathbb C}}
\newcommand{\bbR}{{\mathbb R}}

\newcommand{\cO}{{\mathcal O}}
\newcommand{\sE}{{\sf E}}
\newcommand{\fA}{\mathfrak{A}}
\newcommand{\fH}{\mathfrak{H}}
\newcommand{\fK}{\mathfrak{K}}
\newcommand{\fL}{\mathfrak{L}}
\newcommand{\fM}{\mathfrak{M}}
\newcommand{\fN}{\mathfrak{N}}

\newcommand{\dist}{\mathop{\rm dist}}
\newcommand{\lal}{{\langle}}
\newcommand{\ral}{{\rangle}}

\newcommand{\Dom}{\mathop{\mathrm{Dom}}}

\newcommand{\Ran}{\mathop{\mathrm{Ran}}}
\DeclareMathOperator{\spec}{spec}

\DeclareSymbolFont{SY}{U}{psy}{m}{n}
\DeclareMathSymbol{\emptyset}{\mathord}{SY}{'306}
\newcommand{\R}{{\Bbb R}}
\newcommand{\C}{{\Bbb C}}

\newcommand{\cB}{{\cal B}}

\newcommand{\cG}{{\cal G}}

\newcommand{\cV}{{\cal V}}
\newcommand{\ri}{{\rm i}}

\newcommand{\wY}{\widehat{Y}}
\newcommand{\wZ}{\widehat{Z}}

\newcommand{\Inf}{\mathop{\rm inf}}
\newcommand{\Sup}{\mathop{\rm sup}}

\newcommand{\diag}{\mathop{\rm diag}}
\newcommand{\Img}{\mathop{\rm Im}}
\newcommand{\Real}{\mathop{\rm Re}}

\allowdisplaybreaks

\numberwithin{equation}{section}

\newtheorem{theorem}{Theorem}[section]
\newtheorem{corollary}[theorem]{Corollary}
\newtheorem{lemma}[theorem]{Lemma}

\newtheorem{hypothesis}[theorem]{Hypothesis}
\theoremstyle{definition}

\theoremstyle{remark}
{\it}{\rm}
\newtheorem{remark}[theorem]{Remark}

\begin{document}
\title[ON INVARIANT GRAPH SUBSPACES OF A $J$-SELF-ADJOINT OPERATOR]
{ON INVARIANT GRAPH SUBSPACES OF A $J$-SELF-ADJOINT OPERATOR\\ IN THE FESHBACH CASE$^*$}%
\thanks{$^*$Support of this research by Deutsche Forschungsgemeinschaft,
Russian Foundation for Basic Research, and the Heisenberg-Landau
Program is gratefully acknowledged. The work has also been supported
by St.\,Pe\-ter\-s\-burg State University (grant \#
11.38.241.2015).}

\author[S. Albeverio and A. K. Motovilov]
{Sergio Albeverio and  Alexander K. Motovilov}

\address{Sergio Albeverio,
Institut f\"ur Angewandte Mathematik und HCM, Uni\-ver\-si\-t\"at Bonn,
Endenicher Allee 60, D-53115 Bonn, Germany}
\email{albeverio@uni-bonn.de}

\address{Alexander K. Motovilov, Bogoliubov Laboratory of
Theoretical Physics, JINR, Joliot-Cu\-rie 6, 141980 Dubna, Moscow
Region, Russia} \email{motovilv@theor.jinr.ru}




\maketitle
\thispagestyle{empty}

\begin{quote}
We consider a $J$-self-adjoint  $2\times2$ block operator matrix $L$
in the Feshbach spectral case, that is, in the case where the
spectrum of one main-diagonal entry is embedded into the absolutely
continuous spectrum of the other main-diagonal entry. We work with
the analytic continuation of the Schur complement of a main-diagonal
entry in\, $L-z$\, to the unphysical sheets of the spectral parameter
$z$ plane. We present the conditions under which the continued Schur
complement has operator roots, in the sense of Markus-Matsaev. The
operator roots reproduce (parts of) the spectrum of the Schur
complement, including the resonances. We then discuss the case where
there are no resonances and the associated Riccati equations have
bounded solutions allowing the graph representations for the
corresponding $J$-orthogonal invariant subspaces of $L$. The
presentation ends with an explicitly solvable example.
\end{quote}
\bigskip

\section{Introduction}
\label{SIntro}
In this note we deal with a $2\times2$ block operator matrix of the form
\begin{equation}
\label{L}
L=\left(\begin{array}{rl}
  A_0       &   B  \\
  -B^*     &   A_1
\end{array}\right)
\end{equation}
It is assumed that $A_0$ and $A_1$ are self-adjoint operators in
Hilbert spaces $\fA_0$ and $\fA_1$, respectively, and $B$ is a
bounded operator from $\fA_1$ to $\fA_0$. The block operator matrix
$L$ is understood as an operator in the orthogonal sum
$\fH=\fA_0\oplus\fA_1$ of the Hilbert spaces $\fA_0$ and $\fA_1$,
and $\Dom(L)=\Dom(A_0)\oplus\Dom(A_1)$.

When studying operators of the form \eqref{L}, one usually
introduces the involution
\begin{equation}
\label{J}
J=\left(\begin{array}{cc}
  I       &   0  \\
 0     &   -I
\end{array}\right),
\end{equation}
where $I$ denotes the identity operator. In this case the product
$JL$ is a self-adjoint operator in $\fH$, and then the operator $L$
is called $J$-self-adjoint. Surely, the $J$-self-adjoint operator
\eqref{L} may be viewed as a perturbation, $L=A+V$, of the
block-diagonal self-adjoint operator matrix $A:=\diag(A_0,A_1)$,
$\Dom(A)=\Dom(A_0)\oplus\Dom(A_1)$, by the bounded
$J$-self-adjoint off-diagonal operator block matrix
\begin{equation}
\label{V}
V=\left(\begin{array}{cc}
  0       &   B  \\
 -B^*     &   0
\end{array}\right).
\end{equation}

The involution \eqref{J} induces an indefinite inner product
\begin{equation}
\label{IpKs} [x,y]=(Jx,y), \quad x,y\in\fH.
\end{equation}
Equipping the Hilbert space $\fH$ with the indefinite inner product
\eqref{IpKs} transforms it into a Krein space, which we denote by
$\fK$, $\fK=\{\fH,[\,\cdot\,,\,\cdot\,]\}$. Notice that if an
operator is $J$-self-adjoint in $\fH$ then it is self-adjoint in
$\fK$. In particular, the operator \eqref{L} is self-adjoint in $\fK$.
The theory of linear and, in particular, self-adjoint operators in Krein
spaces is already a deeply developed subject and for the
corresponding definitions, concepts and main results we refer
the reader, e.g., to \cite{Langer82}, \cite{Bognar}, or~\cite{AI}.
For the recent spectral results and further references see, e.g.,
\cite{ABJT} and \cite{PST}

Surely, for $B\neq 0$ the $J$-self-adjoint operator $L$ given by
\eqref{L} can not be self-adjoint in $\fH$ with respect to the
original inner product. Nevertheless, in many cases the spectrum of
such an operator is purely real and, moreover, $L$ turns out to be
similar to a self-adjoint operator. Such a situation takes place if
$L$ possesses a couple of complementary $J$-orthogonal reducing
subspaces $\fL_0$ and $\fL_1$ that are maximal uniformly definite
(respectively, positive and negative) with respect to the Krein
space inner product \eqref{IpKs} (see, e.g., the papers
\cite{AMSh08,AMT09} and references cited therein). Basically, this
happens for the case where the spectra
\begin{equation}
\label{spec}
\sigma_0:=\spec(A_0)\text{\,\, and \,\,}\sigma_1:=\spec(A_1)
\end{equation}
of the entries $A_0$ and $A_1$ are disjoint, i.e.,
\begin{equation}
\label{deldist}
\delta:=\dist(\sigma_0,\sigma_1)>0
\end{equation}
and the norm of $B$ is sufficiently small  (see \cite[Theorem
5.8]{AMSh08} or \cite[Theorem 6.1]{AMT09}): In general, we need to
have $\|V\|<\frac{\delta}{\pi}$ but if the spectral sets $\sigma_0$
and $\sigma_1$ are separated from each other by only one or two
gaps, then the sufficient condition reduces to the bound
$\|V\|<\frac{d}{2}$. Sufficient conditions for the similarity of
a $J$-self-adjoint operator to a self-adjoint one are  also known in
the case of some unbounded $B$ \cite{Veselic1,Veselic2}.

The maximal uniform definiteness of the subspaces $\fL_0$ and
$\fL_1$ suggests the existence of strictly contractive
operators $Y\in\cB(\fA_1,\fA_0)$ and $\wY=Y^*$ (see, e.g.,
\cite[Section 2]{AMT09}) such that $\fL_0$ is the graph of $\wY$
and $\fL_1$ is the graph of $Y$,
\begin{equation}
\label{Lgspaces}
\fL_0=\cG(\wY):=\{x_0\oplus\wY x_0\,|\,\, x_0\in\fA_0\},\quad
\fL_1=\cG(Y):=\{ Y x_1\oplus x_1\,|\,\, x_1\in\fA_1\}.
\end{equation}
The angular operators $Y$ and $\wY$ are strong solutions
for the pair of respective dual operator Riccati equations:
\begin{equation}
\label{RicY}
A_0Y-YA_1+YB^*Y=-B.
\end{equation}
and
\begin{equation}
\label{RicX}
\wY A_0-A_1\wY+\wY B\wY=-B^*.
\end{equation}

In the present work we are concerned with the case where
condition \eqref{deldist} fails to hold from the very beginning:
We assume that the entry $A_0$ has only absolutely continuous spectrum
and that the spectrum of $A_1$ is completely embedded into the spectrum
of $A_0$, that is,
\begin{equation}
\label{sig10}
\sigma_1\subset\sigma_0.
\end{equation}
For the case \eqref{sig10} one knows that, in general, the operator
$L$ has complex spectrum and that the following enclosure holds:
$\spec(L)\setminus\bbR\subset\{z\in\bbC\, |\,\, |\Img z|\leq\|B\|\}$
(see \cite[Theorem 5.5]{Tretter2009}). It is also known that if $A$
is bounded or semibounded then $\inf\,\,\spec(A)\leq
\Real\,\,\spec(L)\leq\sup\,\,\spec(A)$ (see \cite[Theorem
5.8]{AMT09}).

In order to study the spectral problem for the block operator matrix
\eqref{L}, we employ the Frobenius-Schur factorization
(see. e.g., \cite[Proposition 1.6.2 and Theorem
2.2.13]{TretterBook}) of the difference $L-z$, $z\not\in\sigma_0$:
\begin{equation}
\label{SchFa}
L-z=\left(\begin{array}{cc}
I & 0 \\
-B^*(A_0-z)^{-1} & I
\end{array}\right)
\left(\begin{array}{cc}
A_0-z & 0 \\
0 & M_1(z)
\end{array}\right)
\left(\begin{array}{cc}
I & (A_0-z)^{-1}B \\
0 & I
\end{array}\right),
\end{equation}
where $M_1(z)$ stands for the Schur complement of $A_0-z$,
\begin{equation}
\label{M1}
M_1(z)=A_1-z+W_1(z)
\end{equation}
with
\begin{equation}
\label{V1}
W_1(z)=B^*(A_0-z)^{-1}B,
\end{equation}
Notice that the resolvent $(L-z)^{-1}$ can be expressed explicitly
in terms of the inverse $M_1^{-1}(z)$; \eqref{SchFa} also implies
that $\spec(L)\setminus\spec(A_0)\subseteq\spec(M_1)$. Therefore, in
studying the spectral properties of the transfer function $M_1$ one
studies at the same time the spectral properties of the operator
matrix $L$.

Assuming that the absolutely continuous spectrum $\sigma_0$ consist
of the only branch presented by a finite or infinite closed interval
$\Delta_0\subset\bbR$, in Section \ref{SeSchur} we formulate
conditions on $B$ allowing to perform analytic continuation of the
Schur complement $M_1(z)$ through the cut along $\Delta_0$ to
certain domains lying on the neighboring unphysical sheets of
the spectral parameter plane. Here we follow exactly the line of the
work \cite{MennMot} in its simplified version \cite{MM-OTAA}.

Having two variants of the continued Schur complement $M_1$,
produced by crossing the cut $\Delta_0$ from $\bbC^+$ down and from
$\bbC^-$ up, for both of them in Section \ref{Sfactor} we prove the
existence of the respective operator roots $Z^{(l)}$, $l=-1$ and
$l=+1$. The spectrum of the operator root $Z^{(l)}$, $l=\pm1$, is
just the spectrum of the analytically continued Schur complement
$M_1$ lying at some neighborhood of the set $\sigma_1$. The size of
this neighborhood is determined by the strength of the operator $B$.
The spectrum of $Z^{(l)}$ along with (a part of) the spectrum of $L$
may include resonances (by which we understand the complex spectrum
of $M_1$ located in the continuation domain on the corresponding
unphysical sheet).

In Section \ref{SNoRes} we discuss the case where the operator
$Z^{(l)}$, $l=\pm1$, has no real and resonant spectrum. In this
case, under minimal additional assumptions, the operator Riccati
equations \eqref{RicY} and \eqref{RicX} are proven to be solvable.
However, unlike in the cases of disjoint spectral components
$\sigma_0$ and $\sigma_1$ considered in \cite{AMSh08,AMT09}, now the
operator $L$ has complex spectrum and the $L$-invariant graph
subspaces \eqref{Lgspaces} are not maximal uniformly
definite.

Section \ref{SFeshbach} presents an example that just fits the main
assumptions of Section \ref{SNoRes}. Namely, in Section
\ref{SFeshbach} we deal with the spectral disposition that is called
Feshbach --- in complete analogy with the celebrated similar one in
the case of quantum-mechanical Hamiltonians. We assume
that the subspace $\fA_1$ is finite-dimensional and that the
perturbation $B$ is such that it completely sweeps the eigenvalues of
$A_1$ (which are all embedded into the absolutely continuous
spectrum of $A_0$) from the real axis. However, just opposite to the
Hermitian case, the eigenvalues of $A_1$ turn not into resonances
but into the complex spectrum of $L$. The resulting operator roots
$Z^{(l)}$, $l=\pm1$, have neither real no resonance spectrum and,
thus, $J$-self-adjoint block operator matrix \eqref{L} possesses a
couple of mutually $J$-orthogonal invariant graph subspaces of the
form \eqref{Lgspaces}.

Finally, in Section \ref{SSimp} we present the simplest possible
example with $A_0$ being the operator of multiplication by independent
variable in $\fA_0=L_2(-\alpha,\alpha)$, $\alpha>0$, and $A_1=a_1$
being the multiplication by a number $a_1\in(-\alpha,\alpha)$ in
$\fA_1=\bbC$. At least for $a_1=0$, the norm of the corresponding
solutions $Y$ and $\wY=Y^*$ in this example is computed explicitly:
$\|Y\|=\|\wY\|=1$.

The following notations are used thro\-ug\-h\-o\-ut the paper. By
$\bbC^+$ and $\bbC^-$ we understand respectively the upper and lower
half-planes of the complex plane $\bbC$, e.g., $\bbC^+=\{z\in\bbC
\,\,|\,\, \Img z>0\}$. By a subspace of a Hilbert space we always
mean a closed linear subset. The Banach space of bounded linear
operators from a Hilbert space $\fM$ to a Hilbert space $\fN$ is
denoted by $\cB(\fM,\fN)$ and by $\cB(\fM)$ if $\fN=\fM$. The
notation $\sE_T(\sigma)$ is always used for the spectral projection
of a self-adjoint operator $T$ associated with a Borel set
$\sigma\subset\bbR$. By $\cO_r(\sigma)$, $r\geq 0$, we denote the
closed $r$-neigh\-bourhood of $\sigma$ in $\bbC$, i.e.\
$\cO_r(\sigma)=\{z\in\bbC\big|\,\dist(z,\sigma)\leq r\}$. We let
$\Dom(Z)$ and $\Ran(Z)$ denote the respective domain and range of a
linear operator $Z$.

\medskip

\section{Analytic continuation of the Schur complement}
\label{SeSchur}
In this note we restrict our
consideration to the case where all the spectrum $\sigma_0:=\spec(A_0)$ of
the entry $A_0$ is absolutely continuous and coincides with the closure of the single
interval $\Delta_0:=(\mu_0^{(1)},\mu_0^{(2)})$ with
$-\infty\leq\mu_0^{(1)}<\mu_0^{(2)}\leq\infty$. Furthermore, the
whole spectrum $\sigma_1:=\spec(A_1)$ of the entry $A_1$ is assumed to be
embedded into the interval $\Delta_0$, that is, $\sigma_1\subset
\Delta_0$.

Denote by $\sE_0$ the spectral measure of the self-adjoint operator
$A_0$ and let $\sE^0(\mu):=\sE_0\bigl((-\infty,\mu)\bigr)$ be the spectral
function of $A_0$. Then the function $W_1(z)$ can be written as
\begin{equation}
\label{W1z}
    W_1(z)=\displaystyle\int_{\sigma_0}dK_B(\mu)(\mu-z)^{-1}
\end{equation}
where
\begin{equation}
\label{KB}
  K_B(\mu)=B^*\sE^0(\mu)B.
\end{equation}
Our central assumption is that the operator-valued function
$K_B(\mu)$ is differentiable in \mbox{$\mu\in\Delta_0$} in the
operator norm topology and that the derivative $K'_B(\mu)$ admits
analytic continuation from $\Delta_0$ to a simply connected domain
$D^-$ located in $\bbC^-$. We suppose that the boundary of $D^-$
includes the entire spectral interval $\Delta_0$.  Since $K'_B(\mu)$
represents a self-adjoint operator for any $\mu\in\Delta_0$ and
$\Delta_0\subset\R$, the function $K'_B(\mu)$ also admits analytic
continuation from $\Delta_0$ into the domain $D^+$, symmetric to
$D^-$ with respect to the real axis and
$[K'_B(\mu)]^*=K'_B(\bar{\mu}),$ $\mu\in D^\pm\,.$ For the case where the end point
$\mu_0^{(k)}$, $k=1,2$, of the spectral interval $\Delta_0$ is finite
we will always suppose that
\begin{equation}
\label{KBbnd}
\|K'_B(\mu)\|\leq C|\mu-\mu_0^{(k)}|^\gamma,\quad \mu\in D_k^\pm,
\end{equation}
with some $C>0$ and $\gamma\in(-1,0]$.

In the notations like $D^+$ and $D^-$ below we will often use the
number index $l$, $l=+1$ or $l=-1$, identifying the values of $l$ in
the notation $D^l$ with the respective signs ``$+$'' or ``$-$''.

Let $\Gamma^l$, $l=\pm1$, be a rectifiable Jordan curve originating
from a continuous deformation of the interval $\Delta_0$ and lying
in $D^l$, with the (finite) end points fixed. The quantity
\begin{equation}
\label{V0G}
\cV_0(B,\Gamma^l):=\displaystyle\int_{\Gamma^l}|d\mu|\,\|K'_B(\mu)\|
\end{equation}
where $|d\mu|$ stands for the Lebesgue measure on the contour
$\Gamma^l$, is called the variation of the operator-valued function
$K_B(\mu)$ along $\Gamma^l$.  We suppose that there are contours
$\Gamma^l$ with finite $\cV_0(B,\Gamma^l)$ even in the case of
unbounded $\Delta_0$.  Jordan contours $\Gamma^l$ satisfying
condition~$\cV_0(B,\Gamma^l)<\infty$ are said to be \textit{admissible}.

We will need the following elementary statement (cf., e.g.,
\cite[Lemma 2.1]{MennMot}).

\begin{lemma}\label{M1-Continuation}
The analytic continuation of the Schur complement $M_1(z)$,
$z\in\C\setminus{\sigma_0}$, through the spectral interval
$\Delta_0$ into the subdomain $D(\Gamma^l)\subset D^l$,
$l=\pm1$, bounded by $\Delta_0$ and an admissible contour
$\Gamma^l$ is given by
\begin{equation}
\label{Mcmpl}
    M_1(z,\Gamma^l)=A_1-z+W_1(z,\Gamma^l)
    \quad\mbox{\rm with}\quad
W_1(z,\Gamma^l)=
\int_{\Gamma^l}
d\mu\,K'_B(\mu)\,(\mu-z)^{-1}.
\end{equation}
For $z\in D^l\cap D(\Gamma^l)$ one has
\begin{equation}
\label{Sheets}
\mbox{$M_1(z,\Gamma^l)=M_1(z)-2\pi\ri\,l K'_B(z).$}
\end{equation}
\end{lemma}
\begin{proof}
The proof is reduced to the observation
that the function $M_1(z,\Gamma^l)$ is holomorphic for
$z\in\C\setminus\Gamma^l$ and coincides with $M_1(z)$ for
$z\in\C\setminus\overline{D(\Gamma^l)}$.  The last equation
representing $M_1(z,\Gamma^l)$ via $M_1(z)$  is obtained
from~(\ref{Mcmpl}) by using the Residue Theorem.
\end{proof}

\begin{remark}
From the representation \eqref{Sheets} it follows that the Riemann
surface of the Schur complement $M_1(z)$ is  larger than the single
sheet of the spectral parameter plane. The sheet of the complex
plane $\bbC$ where the function $M_1(z)$ and the resolvent
\mbox{$(L-z)^{-1}$} are considered initially is called the physical
sheet. Formula \eqref{Sheets} implies that the domains $D^-$ and
$D^+$ are to be placed on additional  sheets of the $z$ plane that
are different from the physical sheet. We remind that these additional
sheets are usually called unphysical sheets (see, e.g.,
\cite{ReedSimonIII}).  In this work we only deal with the unphysical
sheets attached (through the interval $\Delta_0$) immediately to the
physical sheet.
\end{remark}

\section{A factorization result}
\label{Sfactor}
Suppose that the spectrum of a linear operator $Z\in\cB(\fA_1)$ does
not intersect an admissible contour $\Gamma\subset D^\pm$\,. Then
one can introduce the (transformator) operator
\begin{equation}
\label{W1Z}
W_1(Z,\Gamma)=-\displaystyle
\int_{\Gamma}d\mu\,K'_B(\mu)(Z-\mu)^{-1}.
\end{equation}
Clearly, if the resolvent $(Z-\mu)^{-1}$, $\mu\in\Gamma$, is
uniformly bounded on $\Gamma$, the operator $W_1(Z,\Gamma)$ is also bounded,
\begin{equation}
\label{W1ZG}
\|W_1(Z,\Gamma)\|\leq \cV_0(B,\Gamma)\,
\Sup_{\mu\in\Gamma}\|(Z-\mu)^{-1}\|\,.
\end{equation}
Below, we consider the \emph{basic} equation (cf.~\cite{MotRem})
\begin{equation}
\label{MainZ}
Z=A_1+W_1(Z,\Gamma)
\end{equation}
or, equivalently,
\begin{equation}
\label{MainEqC}
X=W_1(A_1+X,\Gamma).
\end{equation}
assuming that a solution $X$ of the latter is looked for in $\cB(\fA_1)$.

It is worth noting that if $X$ is a solution of~(\ref{MainEqC}) and
$u_1$ is an eigenvector of \mbox{$Z=A_1+X$} corresponding to an
eigenvalue $z$, $Z u_1=zu_1$, then by \eqref{W1Z} and \eqref{MainZ}
we have
\begin{equation}
\label{zZ}
zu_1=A_1
u_1+W_1(Z,\Gamma)u_1=A_1 u_1+W_1(z,\Gamma)u_1.
\end{equation}
Thus,
any eigenvalue $z$ of $Z$ is simultaneously an eigenvalue for the
continued transfer function $M_1(z,\Gamma)$ while $u_1$ is an
eigenvector of this function.

The next statement is a variant of an existence result from
\cite[Theorem 3.1]{MennMot} (cf. similar results in
\cite{HMMJOT,HMMOT}) rewritten for the $J$-self-adjoint case under
consideration.

\begin{theorem}\label{Solvability}
Assume that $\Gamma$ is an admissible contour satisfying the condition
\begin{equation}
\label{Best}
\cV_0(B,\Gamma)< \displaystyle\frac{1}{4}\,d^2(\Gamma)\,,
\end{equation}
where $d(\Gamma):=\dist\bigl(\sigma_1,\Gamma\bigr).$ Then
equation~\eqref{MainEqC} has a unique solution in any closed ball in
$\cB(\fA_1)$ consisting of operators $X$ satisfying $\|X\|\leq r$
with $r_{\rm min}(\Gamma)\leq r<r_{\rm max}(\Gamma)$ where
\begin{equation}
\label{rminmax}
r_{\rm min}(\Gamma)=d(\Gamma)/2- \sqrt{d^2(\Gamma)/4
-\cV_0(B,\Gamma)}, \quad
r_{\rm max}(\Gamma)=d(\Gamma)-\sqrt{\cV_0(B,\Gamma)}. \,\,
\end{equation}
The unique solution $X$ belongs to the smallest ball, i.e.
$\|X\|\leq r_{\rm min}(\Gamma)$.
\end{theorem}
The above assertion is easily proven by making use of Banach's
Fixed Point Theorem (cf.~\cite{MennMot}).  Furthermore, it is
then shown that if the value of $l=\pm1$ is fixed, the solution $X$
does not depend on
a specific contour $\Gamma\subset D^l$ satisfying \eqref{Best}.
Moreover, the bound on the norm of $X$ may be optimized with respect to the admissible
contours $\Gamma^l$ in the form
$\|X\|\leq r_0(B)$ with
\begin{equation}
\label{r0B}
  r_0(B):=\Inf\limits_{\Gamma^l:\,\omega(B,\Gamma^l)>0}
  r_{\rm min}(\Gamma^l)\,,
\end{equation}
where $\omega(B,\Gamma^l)=d^2(\Gamma^l)-4\cV_0(B,\Gamma^l).$
Unlike $r_0(B)$, the solution $X$ depends on $l$, and thus we will
supply its notation with the index $l$ writing $X^{(l)}$. As it is
seen from the next statement, the operators $Z^{(l)}=A_1+X^{(l)}$,
$l=\pm1$, may be understood as operator roots of the continued Schur
complement $M_1$.

Two assertions below (Theorem \ref{factorization} and Theorem
\ref{MHOmega}) may be proven in exactly the same way as the
corresponding statements in \cite{MennMot} (see \cite[Theorems 4.1
and 4.9]{MennMot}; also cf. \cite{MM-OTAA}), only the plus and minus signs interchange in
certain places. Thus, we present these assertions without a proof.

\begin{theorem}\label{factorization}
Let $\Gamma^l$ be an admissible contour satisfying {\rm(\ref{Best})}
and let $Z^{(l)}=A_1+X^{(l)}$ where $X^{(l)}$ is the corresponding
unique solution of \eqref{MainEqC} mentioned in Theorem \ref{Solvability}.
Then, for $z\in\C\setminus\Gamma^l$, the Schur complement
$M_1(z,\Gamma^l)$ admits the following factorization:
\begin{equation}
\label{Mfactor}
    M_1(z,\Gamma^l)=F_1(z,\Gamma^l)\,(Z^{(l)}-z)\,,
\end{equation}
where
\begin{equation}
\label{F1zG}
F_1(z,\Gamma^l)=I+\int_{\Gamma^l}
d\mu\,K'_B(\mu)(Z^{(l)}-\mu)^{-1}(\mu-z)^{-1}\,%
\end{equation}
is a bounded operator on $\fA_1$. Moreover, if
$\dist\bigl(z,\sigma_1\bigr)\leq\frac{1}{2}{d(\Gamma^l)}$ then for
sure $F_1(z,\Gamma^l)$ has a bounded inverse.
\end{theorem}
Following the Markus and Matsaev factorization result for
holomorphic operator-valued functions \cite{MrMt}
we interpret the factorization property
\eqref{Mfactor} in the sense that $Z^{(l)}$ is the operator root of
the analytically continued Schur complement $M_1(\cdot,\Gamma^l)$.

As an elementary consequence of Theorem \ref{factorization}, we obtain
the following corollary.

\begin{corollary}\label{SpHalfVic}
The spectrum $\sigma(Z^{(l)})$ of the operator $Z^{(l)}=A_1+X^{(l)}$
lies in the closed $r_0(B)$-neighborhood ${\cal O}_{r_0}(\sigma_1)$
of the spectrum of $A_1$ in $\bbC$.  Moreover, the spectrum of
$M_1(\,\cdot\,,\Gamma^l)$ lying in the closed
$d(\Gamma)/2$-neighborhood ${\cal O}_{d(\Gamma)/2}(\sigma_1)$ of
$\sigma_1$ in $\bbC$ is nothing but the spectrum of $Z^{(l)}$.
\end{corollary}

Let
\begin{equation}
\label{Oml}
\Omega^{(l)}=\displaystyle\int_{\Gamma^l} d\mu\,
(Z^{(-l)*}-\mu)^{-1}K'_B(\mu)\, (Z^{(l)}-\mu)^{-1}
\end{equation}
where $\Gamma^l$ denotes an admissible contour satisfying the
condition~\eqref{Best}.
\begin{theorem}
\label{MHOmega}
The operators $\Omega^{(l)}$, $l=\pm1$, have the following
properties {\rm(}cf. {\rm\cite{MenShk}):}
\begin{eqnarray}
\label{Om0VG}
\|\Omega^{(l)}\|&<& \frac{\cV_0(B,\Gamma)}{\frac{1}{4}\,d^2(\Gamma)}<1, \qquad\\
\label{OmOm}
\Omega^{(-l)}&=&\Omega^{(l)*}, \\
\label{OmegaM}
-\frac{1}{2\pi\ri}\int_\gamma dz\,[M_1(z,\Gamma^l)]^{-1} &=&
(I-\Omega^{(l)})^{-1}\,,\\
\label{Hadj}
-\frac{1}{2\pi\ri}\int_\gamma dz\,z\,[M_1(z,\Gamma^l)]^{-1} &=&
(I-\Omega^{(l)})^{-1}Z^{(-l)*}=
Z^{(l)}(I-\Omega^{(l)})^{-1}\,\quad
\end{eqnarray}
where $\gamma$ stands for an arbitrary rectifiable closed
contour going in the positive direction around the
spectrum of $Z^{(l)}$ inside the set ${\cO}_{d(\Gamma)/2}(\sigma_1)$.
\end{theorem}

The expressions \eqref{OmegaM} and~\eqref{Hadj} allow us, in
principle, to find the operators $Z^{(l)}$ and, thus, to solve the
equation~\eqref{MainEqC} only by using the contour integration of
the inverse of the continued Schur complement
$[M_1(z,\Gamma^l)]^{-1}$.  From \ref{Hadj}) it also follows that the
operators $Z^{(-1)*}$ and $Z^{(+1)}$ are similar to each other and,
thus, the spectrum of $Z^{(-1)*}$ coincides with that of $Z^{(+1)}$.

\section{Solvability of the operator Riccati equations in case \\
of absence of the real spectrum and resonances}
\label{SNoRes}

To avoid some  purely technical complications, in the rest of the
paper we assume that the operator $A_1$ is bounded.

Recall that we work under the assumption that the spectrum of $A_0$
is absolutely continuous and that it coincides with the closure
$\overline{\Delta}_0$ of the interval $\Delta_0\subset\bbR$. The
interval $\Delta_0$ is a part of the boundary of the continuation
domain $D^l\subset\bbC^l$, $l=\pm1$, for the Schur complement $M_1$.
Suppose it so happened that the spectrum of the operator $Z^{(l)}$
is separated from the spectrum of $A_0$ and that, in addition, the
(complex) spectrum of  $Z^{(l)}$ in the domain $D^l$ is empty. In
other words, let us make the following assumption.

\begin{hypothesis}
\label{HypR} Let the (self-adjoint) operator $A_1$ be bounded. Assume
the hypothesis of Theorem \ref{Solvability} for an admissible
contour $\Gamma^{l}\subset D^l$, $l=\pm1$ and let
$Z^{(l)}=A_1+X^{(l)}$ where $X^{(l)}$ denotes the corresponding
unique solution to \eqref{Best}. Suppose that
\begin{equation}
\label{sZc}
\dist\bigl(\spec(Z^{(l)}),D^l\bigr)>0.
\end{equation}
\end{hypothesis}
\begin{remark}
One notices that under Hypothesis \ref{HypR} and, in particular,
under the assumption \eqref{sZc} the resolvent $(Z^{(l)}-\mu)^{-1}$
is uniformly bounded on $\overline{\Delta}_0$, i.e.
$\sup\limits_{\mu\in\overline{\Delta}_0}\|(Z^{(l)}-\mu)^{-1}\|<\infty$,
and hence
\begin{equation}
\label{Xb}
\int_{\Delta_0} d\mu\|K'_B(\mu)\|\, \|(Z^{(l)}-\mu)^{-1}\|^2<\infty.
\end{equation}
\end{remark}

Under Hypothesis \ref{HypR} one may talk on the existence
of solutions to the operator Riccati equations \eqref{RicX} and
\eqref{RicY}.

\begin{lemma}
\label{LRic}
Assume Hypothesis \ref{HypR} for some $l=\pm1$ and set
\begin{equation}
\label{Ysol}
Y^{(l)}=\int_{\Delta_0} \sE_0(d\mu) B(Z^{(l)}-\mu)^{-1}.
\end{equation}
The operator $Y^{(l)}$ is a bounded operator from $\fA_1$ to $\fA_0$,
\begin{equation}
\label{YsN}
\|Y^{(l)}\|\leq \left(\int_{\Delta_0} d\mu\|K'_B(\mu)\|\, \|(Z^{(l)}-\mu)^{-1}\|^2\right)^{1/2}.
\end{equation}
Moreover, $Y^{(l)}$ is a strong solution to the operator Riccati equation
\eqref{RicY}.
\end{lemma}
\begin{proof}
By the hypothesis, there is no spectrum of $Z^{(l)}$ in $D^l$ and,
hence, on the closure of the subdomain $D(\Gamma^{l})$ bounded by
the interval $\Delta_0$ and the curve $\Gamma^{l}$. In particular,
$\dist\bigl(\Delta_0,\spec(Z^{(l)})\bigr)>0$ and one can transform
the integration contour $\Gamma^{l}$ into $\Delta_0$ and
equivalently replace equation \eqref{MainZ} by the equation
\begin{equation}
\label{MainZR}
Z^{(l)}=A_1-B^*\int_{\Delta_0} \sE_0(d\mu)B(Z^{(l)}-\mu)^{-1}.
\end{equation}
The integrals on the right-hand sides of \eqref{Ysol} and
\eqref{MainZR} are understood in the strong sense.
Clearly, due to \eqref{KB} and \eqref{Xb} we have
\begin{align}
\nonumber
\left\| \int_{\Delta_0}\sE_0(d\mu)B(Z^{(l)}-\mu)^{-1}x \right\|^2&=
\nonumber
\biggl\langle \int_{\Delta_0}  B^*\sE_0(d\mu)B(Z^{(l)}-\mu)^{-1}x,(Z^{(l)}-\mu)^{-1}x \biggr\rangle\\
\nonumber
&=\biggl\langle \int_{\Delta_0} d\mu K_B'(\mu) (Z^{(l)}-\mu)^{-1}x,(Z^{(l)}-\mu)^{-1}x \biggr\rangle\\
\label{Xb1}
& \leq \int_{\Delta_0} d\mu\|K'_B(\mu)\|\, \|(Z^{(l)}-\mu)^{-1}\|^2 \|x\|^2,
\end{align}
which implies \eqref{YsN}. From \eqref{MainZR} it also follows that
\begin{equation}
\label{ZAY}
Z^{(l)}=A_1-B^*Y^{(l)},
\end{equation}
and by \eqref{Ysol} the operator $Y^{(l)}$ itself satisfies the equation
\begin{equation}
\label{Yeq} Y^{(l)}=\int_{\Delta_0} \sE_0(d\mu)
B(A_1-B^*Y^{(l)}-\mu)^{-1}.
\end{equation}
Then \cite[Theorem 5.5]{AM11} (also cf. \cite[Theorem 3.4]{AMM})
applies and one concludes that $Y^{(l)}$ is a strong solution to
\eqref{RicY}, completing the proof.
\end{proof}
\begin{corollary}
\label{Cor1} The operator $\wY^{(l)}={Y^{(l)}}^*$ is a strong
solution to the operator Riccati equation \eqref{RicX}. The graphs
\begin{equation}
\label{GLs}
\fL_0^{(l)}=\cG(\wY^{(l)}) \text{\, and \,} \fL_1^{(l)}=\cG(Y^{(l)})
\end{equation}
are mutually $J$-orthogonal invariant subspaces
of the block operator matrix $L$. That is,
\begin{equation}
\label{Jxy0}
\lal Jx,y\ral=0 \quad\text{for any \,}x\in\cG(\wY^{(l)}), \,\, y\in\cG(Y^{(l)})
\end{equation}
and $Lx\in\cG(\wY^{(l)})$ whenever  $x\in\Dom(L)\cap\cG(\wY^{(l)})$
and $Ly\in\cG(Y^{(l)})$ whenever $y\in\Dom(L)\cap\cG(Y^{(l)})$.
\end{corollary}

\begin{remark}
\label{RemYg1}
The following inequality holds:
\begin{equation}
\label{NYg1}
\|Y^{(l)}\|\geq 1\quad (\text{and, hence, \,}\|\wY^{(l)}\|=\|{Y^{(l)}}\|\geq 1),\,\,\, l=\pm1.
\end{equation}
\end{remark}
\begin{proof}
Suppose the opposite, that is, $Y^{(l)}<1$. Then by \cite[Theorem 5.3]{AMSh08}
(see also \cite[Theorem 4.1]{AdL} and \cite[Theorem 3.2]{MenShk})
the operator matrix $L$ is similar to a self-adjoint operator and, hence,
the spectrum of $L$ is purely real. Moreover, in such a case the spectrum $L$ is given by
\begin{equation}
\label{sLZZ}
\spec(L)=\spec(Z^{(l)})\cup\spec(\wZ^{(l)}),
\end{equation}
where
\begin{equation}
\label{wZ}
\wZ^{(l)}:=A_0+B\wY^{(l)}=A_0+B{Y^{(l)}}^*, \quad \Dom(\wZ^{(l)})=\Dom(A_0).
\end{equation}
Thus, in particular, $\spec(Z^{(l)})\subset\bbR$.
But this is not the case since Corollary \ref{SpHalfVic} and Hypothesis \ref{HypR}
imply the inclusion $\spec(Z^{(l)})\subset D^{(-l)}$, i.e.
the set $\spec(Z^{(l)})$ is purely complex. This contradiction completes the proof.
\end{proof}

Thus, under Hypothesis \ref{HypR} the angular operators $Y^{(l)}$,
$l=\pm1$, are definitely not strict contractions, and it is possible
that $1$ is eigenvalue for ${Y^{(l)}}^*Y^{(l)}$ and, hence, for
${\wY^{(l)*}} \wY^{(l)}=Y^{(l)}{Y^{(l)}}^*$. In this case the
invariant graph subspaces \eqref{GLs} of $L$ have a non-trivial
intersection, {$\fL_0^{(l)}\cap\fL_1^{(l)}\neq\{0\}$}.
The equality $\fL_0^{(l)}\cap\fL_1^{(l)}=\{0\}$ and the linear independence
of the invariant graph subspaces \eqref{GLs} is ensured provided it is
known that
\begin{equation}
\label{1YY}
\text{$1\not\in\spec({Y^{(l)}}^*Y^{(l)})$
(and, hence, $1\not\in\spec(Y^{(l)}{Y^{(l)}}^*)$)}.
\end{equation}
In the latter case we would have two versions of the direct decomposition
(see, e.g., \cite[Lemma 2.6 and Remark 2.7]{AMSh08})
\begin{equation}
\label{Hsum}
\fH=\fL_0^{(l)}\dotplus\fL_1^{(l)}, \quad l=\pm1.
\end{equation}
Furthermore, with respect to the decomposition
\eqref{Hsum} the operator $L$ would read as
the block diagonal matrix (see, e.g., \cite[Corollary 2.9]{AMSh08})
\begin{equation}
\label{LZZ}
Z=\diag(\wZ^{(l)},Z^{(l)}), \quad l=\pm1.
\end{equation}

What is a criterion for the situation \eqref{1YY} to take place is
an open problem. In the explicitly solvable example discussed below
in Section \ref{SSimp} we have just the opposite situation: $1$ is an
eigenvalue of both ${Y^{(+)}}^*Y^{(+)}$ and
${Y^{(-)}}^*Y^{(-)}$.

\section{Feshbach case}
\label{SFeshbach}
In the present section we consider the spectral situation that
resembles the Feshbach one in the case of self-adjoint block
operator matrices. Namely, we assume that the spectrum of the
self-adjoint operator $A_1$ only consists of isolated eigenvalues
of finite multiplicities and all these eigenvalues are embedded
into the absolutely continuous spectrum of the self-adjoint operator
$A_0$. If $A_1$ is bounded then the Hilbert space $\fA_1$ is
necessarily finite-dimensional and $A_1$ is finite rank.

We start the discussion of this case with a remark that the
operator-valued function $K_B(\mu)$ for $\mu\in\bbR$ is
non-decreasing. Hence, under our assumptions on analytic properities
of $K_B$, the derivative $K'_B(\mu)$ is a non-negative operator on
$\Delta_0$,
\begin{equation}
\label{Kbp}
\lal K'_B(\mu) x,x\ral\geq 0, \text{ for any }\mu\in\Delta_0 \,\,\text {and any }x\in\fA_1.
\end{equation}

To simply future references, we adopt the following hypothesis.

\begin{hypothesis}
\label{hyp0} Let $\dim(\fA_1)<\infty$ (and, hence, the spectrum
$\sigma_1$ of the operator $A_1$ consists only of a finite number of
isolated eigenvalues of finite multiplicities). Suppose that
$\sigma_1\subset\Delta_0$ and that there is an admissible contour
$\Gamma\subset D^\pm$ such that
$\cV_0(B,\Gamma)<\frac{1}{4}d^2(\Gamma)$. Furthermore, assume that
the derivative $K'_B(\mu)$ is strictly positive and uniformly
bounded from below on the intersection
$\Delta_0\cap\cO_{r_0}(\sigma_1)$ of the interval $\Delta_0$ with
the $r_0$-neighborhood of the spectrum of $A_1$, where $r_0\equiv
r_0(B)>0$ is given by \eqref{r0B}. That is, assume there is $c_0>0$
such that
\begin{equation}
\label{semb} \lal K'_B(\mu) x,x\ral\geq c_0\|x\|^2 \quad \text{for any
}\mu\in\Delta_0\cap\cO_{r_0}(\sigma_1) \,\,\text {and any }x\in\fA_1.
\end{equation}
\end{hypothesis}

Notice that the function $W_1(z)$ given by \eqref{W1z} is Herglotz. One
easily verifies that the following limiting equalities hold:
\begin{align}
\label{W1zlim}
\mathop{\rm Im} \lal W_1(\lambda\pm\ri
0)x,x\ral&:=\lim\limits_{\varepsilon\downarrow 0}\lal
W_1(\lambda\pm\ri \varepsilon)x,x\ral=\pm\pi\lal K'_B(\lambda)
x,x\ral,\\
\nonumber
&\quad \text{for any}\,\,\lambda\in\Delta_0\,\,\text {and any
}x\in\fA_1.
\end{align}

\begin{lemma}
\label{lem1} Assume Hypothesis \ref{hyp0} and let $\Gamma^l$,
$l=\pm1$, be an admissible contour satisfying condition
$\cV_0(B,\Gamma^l)<\frac{1}{4}d^2(\Gamma^l)$. Let $X^{(l)}$ be the corresponding unique
solution of \eqref{MainEqC} mentioned in Theorem \ref{Solvability}
and set $Z^{(l)}=A_1+X^{(l)}$. The spectrum of $Z^{(l)}$ only
consists of finite number of isolated eigenvalues of finite
(algebraic) multiplicities and none of these eigenvalues is real.
\end{lemma}
\begin{proof}
Since the space $\fA_1$ is finite-dimensional, the spectrum of
$Z^{(l)}$ is automatically formed only of isolated eigenvalues with
finite algebraic multiplicities and the number of these eigenvalues
is finite.

Suppose that $u_1\in\fA_1$, $\|u\|=1$, is an eigenvector of
$Z^{(l)}$, $l=\pm1$, corresponding to an eigenvalue $z$, i.e.
$Z^{(l)}u_1=zu_1$. From \eqref{zZ} it follows that
\begin{equation}
\label{zZl}
z=\lal A_1u_1,u_1\ral+\lal W_1(z,\Gamma^l)u_1,u_1\ral.
\end{equation}
One proves that $\Img z\neq 0$ by contradiction. Indeed, assume the opposite, i.e. that $z=\lambda\in\bbR$.
Clearly, we have
\begin{equation}
\label{eqz}
W_1(\zeta,\Gamma^l)=W_1(\zeta)\quad\text{whenever }  \zeta\in\bbC^{-l}, \quad l=\pm1.
\end{equation}
Combining \eqref{eqz} with \eqref{W1zlim} one observes that
\begin{equation}
\label{eqz1}
\mathop{\rm Im} \lal W_1(\lambda,\Gamma^l)u_1,u_1\ral
=\mp l\pi\lal K'_B(\lambda)u_1,u_1\ral\quad
 \text{for any}\,\,\lambda\in\Delta_0.
\end{equation}
In view of Corollary \ref{SpHalfVic},
from \eqref{semb} and \eqref{eqz1} it follows
that
\begin{equation}
\label{eqz2}
|\mathop{\rm Im} \lal W_1(\lambda,\Gamma^l)u_1,u_1\ral|\geq \pi c_0>0.
\end{equation}
Together with $\lal A_1u_1,u_1\ral\in\bbR$ this means that for
$z=\lambda\in\bbR$ the equality \eqref{zZl} is impossible. Thus one
concludes that $Z^{(l)}$ has only non-real eigenvalues, completing
the proof.
\end{proof}

\begin{remark}
\label{remB} Given $l=\pm1$, there is an open neighborhood of the
interval $\bigl(\min\sigma_1,\max\sigma_1\bigr)$ in $\bbC$ that
contains no spectrum of $Z^{(l)}$. This follows from the fact that,
by the continuity argument, at some complex neighborhood of the set
$\cO_{r_0}(\sigma_1)\cap\Delta_0$ the imaginary part of
$W_1(z,\Gamma_1)$ should remain uniformly definite, keeping the same
respective sign that it had on
$\cO_{r_0}(\sigma_1)\cap\Delta_0$.
\end{remark}

In the Feshbach case under consideration, the spectrum of $Z^{(l)}$
represents a part of the ``usual'' spectrum of the block operator
matrix $L$. In other words, unlike in the case of self-adjoint
off-diagonal $V$ in \cite{MennMot,MM-OTAA}, the spectrum of
$Z^{(l)}$ contains no resonances. This is established in the
following lemma.

\begin{lemma}
\label{RemS} Under Hypothesis  \ref{hyp0}, the operator $Z^{(l)}$,
$l=\pm1$, has no spectrum in the corresponding complex domain $D^l$
and, thus, all the eigenvalues of $Z^{(l)}$ are simultaneously
eigenvalues of the original (not yet continued) Schur complement
$M_1(\cdot)$ and, hence, the eigenvalues of $L$.
\end{lemma}
\begin{proof}
Consider the path of $J$-self-adjoint operators $L_t=A+t V$,
$t\in[0,1]$, $\Dom(L_t)=\Dom(A)$. With this path we associate the
corresponding path of the (unique) solutions $X_t^{(l)}$ (from Theorem \ref{Solvability}) to the
respective transformator equations
\begin{equation}
\label{MainEqt}
X=t^2W_1(A_1+X,\Gamma^l).
\end{equation}
as well as the path $Z^{(l)}_t=A_1+X_t^{(l)}$,\,\, $t=[0,1]$.
Obviously, we have $r_0(tB)\to 0$ as $t\to 0$ where the radius $r_0$
is given by \eqref{r0B}. Using the same reasoning as in the proof of
Lemma \ref{lem1}, by Corollary \eqref{SpHalfVic} we then conclude
that $\spec(Z^{(l)}_t)$ lies in $\bbC^{-l}$ and hence
\begin{equation}
\label{ZlD}
\spec(Z^{(l)}_t)\cap D^l=\emptyset
\end{equation}
at least for sufficiently small $t$. Furthermore, by Remark
\ref{remB} we have the strict separation of the spectrum of
$Z^{(l)}_t$ from the real axis at every $0<t\leq 1$.
The solution $X_t^{(l)}$ is a real analytic and, hence, norm continuous
in $t\in[0,1]$. For varying $t$, we then apply the result
on the continuity of finite systems of eigenvalues (see
\cite[Section IV.5]{Kato}) to conclude that the eigenvalues of
$Z_t^{(l)}$ can not jump from $\bbC^{-l}$ to $D^l$.
\end{proof}

\section{The simplest example}
\label{SSimp}
In this section we consider the operator
matrix $L$ of the form \eqref{L} where $A_0$
is the operator of multiplication by independent variable,
\begin{equation}
\label{A0mult}
 (A_0 u_0)(\mu)=\mu\, u_0(\mu),
\end{equation}
on $\fA_0=L_2(-\alpha,\alpha)$,\, $0<\alpha<+\infty$. The spectrum of $A_0$ is
absolutely continuous and fills the interval $[-\alpha,\alpha]$. We assume
that $\fA_1=\bbC$ and, thus, $A_1$ is the multiplication by
a real number $a_1$
\begin{equation}
\label{A1mult}
A_1 u_1=a_1 u_1, \quad u_1\in\bbC.
\end{equation}
The latter should lie inside the continuous spectrum of $A_1$, i.e.
$a_1\in(-\alpha,\alpha)$.
The coupling operator $B:\fA_1\to\fA_0$ is assumed to be
the multiplication by another constant $b\geq 0$, namely
\begin{equation}
\label{Bb}
(Bu_1)(\mu):=bu_1,\quad \mu\in[-\alpha,\alpha]
\end{equation}
(that is, $Bu_1$ is constant function on $[-\alpha,\alpha]$). Obviously, the adjoint operator $B^*$
is given by
\begin{equation}
\label{Bst}
B^*u_0=b\int_{-\alpha}^{\alpha} u_0(\mu)d\mu.
\end{equation}

The self-adjoint analog of such an operator matrix $L$ (with the
lower left entry in \eqref{L} replaced by simply $B^*$) as
well as somewhat more complex self-adjoint operator matrices have
been discussed in detail in \cite[Section 8]{MennMot}. Notice that
the self-adjoint analog of $L$ represents a particular case of one
of the celebrated Friedrichs models \cite{Fried}.

The spectral function $\sE^0(\mu)$, $\mu\in\bbR$, of the
multiplication operator $A_0$ is given by (see,
e.\,g.,~\cite{BirmanSolomiak})
\begin{equation}
\label{MeasureMult}
\bigl(\sE^0(\mu)u_0\bigr)(\nu)=
\chi_{(-\infty,\mu)\cap[-\alpha,\alpha]}(\nu)u_0(\nu), \quad u_0\in\fA_0,
\,\,\nu\in[-\alpha,\alpha],
\end{equation}
where $\chi_{\delta}(\cdot)$ denotes the characteristic function
(indicator) of a Borel set $\delta\subset\bbR$; $\chi_{\delta}(\mu)=1$ if $\mu\in\delta$
and $\chi_{\delta}(\mu)=0$ if $\mu\in\bbR\setminus\delta$.
Hence, the product \eqref{KB} for $-\alpha\leq\mu\leq \alpha$
reads
$$
    K_B(\mu)=b^2\int_{-\alpha}^\mu d\nu=b^2(\mu+\alpha),
$$
and the derivative $K'_B(\mu)$, $\mu\in(-\alpha,\alpha)$, is simply the
constant, $K'_B(\mu)=b^2$, admitting analytic continuation anywhere
on the complex plane. Thus, in the case under consideration one can
choose the whole half-plane $\bbC^+$ as $D^+$ and the whole
half-plane $\bbC^-$ as $D^-$.

The Schur complement \eqref{M1} reads
\begin{equation}
\label{M1in}
M_1(z)=a_1-z+b^2\int_{-\alpha}^{\alpha} \frac{d\mu}{\mu-z} \quad\text{for \,} z\in\bbC\setminus[-\alpha,\alpha],
\end{equation}
while its values $M_1(\lambda\pm\ri 0)$,\, $\lambda\in(-\alpha,\alpha)$,\, on the
rims of the cut are defined as the respective limits in
$z=\lambda\pm\ri\varepsilon$ as $\varepsilon\searrow 0$. The
corresponding continuations \eqref{Mcmpl} are given by
\begin{equation}
\label{M1e} M_1(z,\Gamma^l)=a_1-z+b^2\int_{\Gamma^l}
\frac{d\mu}{\mu-z}, \quad z\in\bbC\setminus\Gamma^l,\quad l=\pm1.
\end{equation}
In this case the basic equation \eqref{MainZ} coincides with the equation
$M_1(z,\Gamma^\pm)=0$ and the solutions $Z^{(\pm)}$ if they
exist are simply the numbers, $Z^{(\pm)}=z^{(\pm)}\in\bbC$.
One easily verifies by inspection that the function $M_1(z)$ does not have
real roots $z\in\bbR\setminus[-\alpha,\alpha]$ and, surely, the same holds true
for the functions \eqref{M1e}.
Furthermore, for $b>0$ none of the functions \eqref{M1in} and \eqref{M1e} has roots in $(-\alpha,\alpha)$
since
$$
\mathop{\rm Im}  M_1(\lambda,\Gamma^\pm)=\mathop{\rm Im}  M_1(\lambda\mp{\rm i}0)=\mp\pi  b^2\quad\text{whenever\, }
\lambda\in(-\alpha,\alpha).
$$
Finally, one notices that $z=-\alpha$ and $z=\alpha$ are branching
points for the Riemann surface of $M_1(\cdot)$.


For the remaining part of the section we set
\begin{equation}
\label{a1eq0}
a_1=0
\end{equation}
and as
${\Gamma^\pm}$ take semi-circumferences, ${\Gamma^\pm}=\{z:\,
|z|=\alpha,\, z\in{\C}^\pm\}$. Then, obviously,
$d:=d(\Gamma^\pm)=\alpha$ and $\cV_0(B,\Gamma^\pm)=\pi\,b^2
\alpha$. The condition \eqref{Best} acquires the form
\begin{equation}
\label{Bestb}
b^2<\frac{\alpha}{4\pi}
\end{equation}
and, given the sign $l=\pm1$, Theorem \ref{Solvability} ensures
the existence of a unique solution $Z^{(l)}=z^{(l)}$ to the
basic equation  \eqref{MainZ} in the open circle $|z|<r_{\rm
max}$,  $z\in{\C}$, where
\begin{equation}
\label{rmax}
r_{\rm max} = d-\sqrt{\cV_0(B,\Gamma^\pm)}=
\alpha-\sqrt{\pi b^2\alpha }>\displaystyle\frac{\alpha}{2}.
\end{equation}
Theorem \ref{Solvability} also guarantees that, in
fact, the unique solution $Z^{(l)}=z^{(l)}$ belongs to the smaller
circle $|z|\leq r_{\rm min}$, $z\in{\C}$, where
\begin{equation}
\label{rmin}
r_{\rm min}=\displaystyle\frac{d}{2}-
\sqrt{\displaystyle\frac{d^2}{4}-\cV_0(B,\Gamma^\pm)}=
\displaystyle\frac{\pi b^2}{\frac{1}{2}+\sqrt{\frac{1}{4}-\frac{\pi b^2}{\alpha}}}
<\displaystyle\frac{\alpha}{2}.
\end{equation}
By Lemmas \ref{lem1} and \ref{RemS} one then concludes that
$\Img z^{(+)}< 0$ and $\Img z^{(-)}>0$ and, thus, that $z^\pm$ are the roots
of the original (not continued) Schur complement $M_1(\cdot)$ given by \eqref{M1in}.
Under the assumption \eqref{a1eq0}, an elementary inspection shows that, for the function  $M_1(\cdot)$,
the roots $z^{(\pm)}$ are purely imaginary,
\begin{equation}
\label{zmpy}
z^{(-)}=-z^{(+)}=\ri y \quad\text{with \,} y>0.
\end{equation}
Moreover, each of these two roots is unique for the upper and lower
half-planes $\bbC^+$ and $\bbC^-$, respectively. Obviously, the
equation $M_1(z^{(\mp)})=0$ with $M_1(\cdot)$ given by \eqref{M1in}
for $a_1=0$ and  $b>0$ is equivalent to
\begin{equation}
\label{M1y1}
1=b^2\int_{-\alpha}^\alpha \frac{d\mu}{\mu^2+y^2},\quad y>0.
\end{equation}
Evaluating the integral in \eqref{M1y1} one finds that $y$ is
the unique positive solution to the equation
\begin{equation}
\label{yatn}
y=2b^2 \arctan\frac{\alpha}{y}
\end{equation}
and for the existence of this solution no smallness requirement like \eqref{Bestb}
is needed: In fact, the unique positive solution to \eqref{yatn} and, hence, the corresponding
unique roots \eqref{zmpy} to the Schur complement \eqref{M1in} exist for any $\alpha>0$ and
$b^2>0$.

According to \eqref{Ysol} and \eqref{MeasureMult}, for the angular
operators $Y^{(l)}$, $l=\pm1$, associated with the above two
solutions $Z^{(l)}=z^{(l)}$ we have
\begin{equation}
\label{Yz}
(Y^{(\pm)}u_1)(\mu)=\frac{b}{z^{(\pm)}-\mu}u_1=-\frac{b}{\mu\pm\ri
y}u_1 \quad\text{for any \,}\mu\in[-\alpha,\alpha]\text{\, and \, }u_1\in\fA_1=\bbC,
\end{equation}
where $y$ is the unique positive solution of \eqref{yatn}. Obviously, from \eqref{Yz} for the norm
of $Y^{(l)}$ one infers
\begin{equation}
\label{Ynorm}
\|Y^{(l)}\|=b\left(\int_{-\alpha}^\alpha \frac{d\mu}{\mu^2+y^2}\right)^{1/2}=1, \quad l=\pm1,
\end{equation}
where the equality \eqref{M1y1} has been taken into account at the last step.
\medskip

\baselineskip=12pt



\end{document}